\documentclass[a4paper,11pt]{article}
\usepackage{amsmath, amsfonts, amscd, amssymb, amsthm, enumerate}

\DeclareMathOperator{\Mat}{\operatorname{M}}

\DeclareMathOperator{\Mata}{\operatorname{A}}
\DeclareMathOperator{\End}{\operatorname{End}}

\DeclareMathOperator{\NT}{\operatorname{NT}}
\DeclareMathOperator{\GL}{\operatorname{GL}}

\DeclareMathOperator{\Vect}{\operatorname{span}}
\DeclareMathOperator{\im}{\operatorname{Im}}

\DeclareMathOperator{\rk}{\operatorname{rk}}

\DeclareMathOperator{\codim}{\operatorname{codim}}
\renewcommand{\setminus}{\smallsetminus}
\renewcommand{\epsilon}{\varepsilon}

\def\F{\mathbb{F}}

\def\calA{\mathcal{A}}

\def\calM{\mathcal{M}}
\def\calN{\mathcal{N}}

\def\calS{\mathcal{S}}
\def\calT{\mathcal{T}}
\def\calU{\mathcal{U}}

\def\calW{\mathcal{W}}
\def\calX{\mathcal{X}}
\def\calY{\mathcal{Y}}

\def\lcro{\mathopen{[\![}}
\def\rcro{\mathclose{]\!]}}

\theoremstyle{definition}
\newtheorem{Def}{Definition}[section]

\theoremstyle{plain}
\newtheorem{theo}{Theorem}[section]
\newtheorem{prop}[theo]{Proposition}
\newtheorem{cor}[theo]{Corollary}
\newtheorem{lemma}[theo]{Lemma}
\newtheorem{claim}{Claim}[section]

\theoremstyle{plain}

\theoremstyle{remark}

\title{On affine spaces of rectangular matrices with constant rank}
\author{Cl\'ement de Seguins Pazzis\footnote{Universit\'e de Versailles Saint-Quentin-en-Yvelines, Laboratoire de Math\'ematiques
de Versailles, 45 avenue des \'Etats-Unis, 78035 Versailles cedex, France}
\footnote{e-mail address: clement.de-seguins-pazzis@ac-versailles.fr}}

\begin{document}

\thispagestyle{plain}

\maketitle
\begin{abstract}
Let $\F$ be a field, and $n \geq p \geq r>0$ be integers.
In a recent article, Rubei has determined, when $\F$ is the field of real numbers,
the greatest possible dimension for an affine subspace of $n$--by--$p$ matrices with entries in $\F$
in which all the elements have rank $r$. In this note, we generalize her result to an arbitrary
field with more than $r+1$ elements, and we classify the spaces that reach the maximal dimension as a function
of the classification of the affine subspaces of invertible matrices of $\Mat_s(\F)$ with dimension
$\dbinom{s}{2}$. The latter is known to be connected to the classification of
nonisotropic quadratic forms over $\F$ up to congruence.
\end{abstract}




\section{Introduction}

\subsection{The problem}\label{section:intro}

Let $\F$ be an arbitrary field. For positive integers $n$ and $p$, denote by $\Mat_{n,p}(\F)$ the vector space of
all matrices with $n$ rows, $p$ columns and entries in $\F$; set $\Mat_n(\F):=\Mat_{n,n}(\F)$ and
denote by $\GL_n(\F)$ its group of invertible elements.
Two subsets $\calX$ and $\calY$ of $\Mat_{n,p}(\F)$ are called \textbf{equivalent}, and we write
$\calX \sim \calY$, whenever there exist $P \in \GL_n(\F)$ and $Q \in \GL_p(\F)$ such that
$\calY=P \calX Q$, i.e.\ $\calX$ and $\calY$ represent the same set of linear mappings in a different choice of bases.
This defines an equivalence relation on affine subspaces of $\Mat_{n,p}(\F)$.

Let $r \in \lcro 1,\min(n,p)\rcro$.
We consider the following three problems:
\begin{enumerate}[(1)]
\item What is the greatest possible dimension $d^{=}_r(n,p)$ for an affine subspace of $\Mat_{n,p}(\F)$ in which every matrix has rank $r$?
\item What is the greatest possible dimension $d^{\geq}_r(n,p)$ for an affine subspace of $\Mat_{n,p}(\F)$ in which every matrix has rank at least $r$?
\item What is the greatest possible dimension $d^{\leq}_r(n,p)$ for an affine subspace of $\Mat_{n,p}(\F)$ in which every matrix has rank at most $r$?
\end{enumerate}
In each case, we can also inquire about the structure of the affine subspaces that attain the greatest possible dimension, an even more difficult question.
By transposing, we see that it suffices to consider the case where $n \geq p$, and from now on we will systematically make that assumption.

Problem (3) has already been solved in \cite{dSPaffpres} (see also \cite{dSPlargerankrevisited} for a more recent account), over all fields,
including not only the description of the spaces that attain the dimension $d^{\leq}_r(n,p)$, but
of the spaces with dimension close to $d^{\leq}_r(n,p)$.
As a reminder, we simply have $d^{\leq}_r(n,p)=nr$, and except when $n=p=|\F|=2$ and $r=1$, the affine subspaces of $\Mat_{n,p}(\F)$ with rank at most $r$ and dimension $nr$ are actually linear subspaces.
Problem (2) has also been solved, with a very special exception: It was proved in \cite{dSPgivenrank} that
$d^{\geq}_r(n,p)=np-\dbinom{r+1}{2}$ (here it is actually better to think in terms of codimension than in terms of dimension);
in \cite{dSPlargeaffinenonsingular} the spaces that have the maximal possible dimension were classified provided $|\F|>2$ and $n=p=r$,
and in \cite{dSPlargeaffinerankbelow} the result was extended to all possible values of $n,p,r$ (still provided that $|\F|>2$).
We will recall the results of the case $n=p=r$ in Section \ref{section:optimal} but their knowledge is not required to understand the present article.

In contrast, problem (1) did not receive any satisfying general answer yet.
In \cite{Rubei}, Rubei started to contribute to it by obtaining the value of $d^{=}_r(n,p)$ for the field of real numbers:

\begin{theo}[Rubei (2022)]\label{theo:Rubei}
Assume that $\F$ is the field of real numbers. Let $n \geq p \geq r>0$ be integers. Then
$$d^{=}_r(n,p)=\dbinom{r}{2}+r(n-r).$$
\end{theo}

To understand why the dimension $\dbinom{r}{2}+r(n-r)$ is optimal, it suffices to consider the affine subspace consisting of all the matrices of the form
$$\begin{bmatrix}
I_r +T & [0]_{r \times (p-r)} \\
[?]_{(n-r) \times r} & [0]_{(n-r) \times (p-r)}
\end{bmatrix}$$
where $T$ ranges over the linear subspace $\NT_r(\F)$ of all strictly upper-triangular matrices of $\Mat_r(\F)$, and the block with the question mark is unspecified.

Rubei abstained from trying to analyse the spaces of maximal dimension for the special case she studied in \cite{Rubei}
(the one of the field of real numbers).

\subsection{Main results}

It is our ambition here to generalize Theorem \ref{theo:Rubei} to almost all fields, and to analyze the spaces with maximal dimension.
Here is our first main result:

\begin{theo}\label{theo:dim}
Let $n \geq p \geq r>0$ be integers. Assume that $|\F|>r+1$.  Then
$$d^{=}_r(n,p)=\dbinom{r}{2}+r(n-r).$$
\end{theo}

Note that the case $p=r$ does not require $|\F|>r+1$ and is a special case of theorem 8 of \cite{dSPgivenrank}.

In the remainder of this introduction, we systematically assume that $|\F|>r+1$.
Further, we will construct a wide variety of spaces that have the maximal dimension. The following terminology will be useful:

\begin{Def}
Let $s>0$. An affine subspace of $\Mat_s(\F)$ is called \textbf{optimal} whenever it is included in $\GL_s(\F)$
(i.e.\ it has constant rank $s$) and its dimension equals $\dbinom{s}{2}$.
\end{Def}

Note that it was proved in \cite{dSPgivenrank} that the greatest possible dimension for an affine subspace of $\Mat_s(\F)$
with constant rank $s$ is $\dbinom{s}{2}$, over all fields (this result is recalled in Section \ref{section:technical} below).
An example of optimal affine subspace of $\Mat_s(\F)$ is $I_s+\NT_s(\F)$.

Now, if we take such an optimal affine subspace $\calW$ of $\Mat_r(\F)$, we can consider the extended space
$$\widetilde{\calW}^{(n,p)}:=\left\{\begin{bmatrix}
W & [0]_{r \times (p-r)} \\
B & [0]_{(n-r) \times (p-r)}
\end{bmatrix} \mid W \in \calW, \; B \in \Mat_{n-r,r}(\F)\right\}.$$
Obviously, all the elements of $\widetilde{\calW}^{(n,p)}$ have rank $r$, and
$\widetilde{\calW}^{(n,p)}$ has dimension $\dbinom{r}{2}+r(n-r)=d^{=}_r(n,p)$.

More generally, consider a partition $r=s+t$ with $s\geq 0$ and $t\geq 0$, and respective optimal subspaces
$\calM$ and $\calN$ of $\Mat_t(\F)$ and $\Mat_s(\F)$. Consider then the space
$\calM \wedge_{n,p} \calN$ of all matrices of the form
$$\begin{bmatrix}
[?]_{s \times t} & N & [?]_{s \times (p-r)} \\
M & [0]_{t \times s} & [0]_{t \times (p-r)} \\
[?]_{(n-r) \times t} & [0]_{(n-r) \times s} & [0]_{(n-r) \times (p-r)} \\
\end{bmatrix} \quad \text{with $M \in \calM$ and $N \in \calN$.}$$
Clearly, all the matrices in this affine subspace have rank $s+t=r$, and this space has dimension
$$st+\dbinom{s}{2}+\dbinom{t}{2}+s\,(p-r)+(n-r)\,t=\dbinom{r}{2}+s\,(p-r)+(n-r)\,t.$$
If $n=p$, this dimension is simply $\dbinom{r}{2}+(s+t)\,(n-r)=\dbinom{r}{2}+r(n-r)=d_r^{=}(n,p)$,
but if $n>p$, $s>0$ and $t>0$ then it is easily checked that this dimension is less than $d_r^{=}(n,p)$.

Now, we can state our second main result:

\begin{theo}\label{theo:maxdim}
Assume that $|\F|>r+1$. Let $n \geq p > r>0$ be integers. Let $\calS$ be an affine subspace of
$\Mat_{n,p}(\F)$ of dimension $d^{=}_r(n,p)$ in which all the matrices have rank $r$.
\begin{enumerate}[(a)]
\item If $n>p$, then $\calS \sim \widetilde{\calW}^{(n,p)}$ for an optimal
affine subspace $\calW$ of $\Mat_r(\F)$, and the equivalence class of $\calW$ is uniquely determined by the one of
$\calS$.
\item If $n=p$ there exists a partition $r=s+t$, with $s \geq 0$ and $t \geq 0$, and
respective optimal subspaces $\calM \subseteq \Mat_t(\F)$ and $\calN \subseteq \Mat_s(\F)$ such that
$\calS \sim \calM \wedge_{n,n} \calN$; Moreover, the pair $(s,t)$ and the equivalence classes of $\calM$ and $\calN$ are uniquely
determined by the equivalence class of $\calS$.
\end{enumerate}
\end{theo}

Here, we exclude the case $p=r$ because it is already known (see \cite{dSPlargeaffinerankbelow} and also Theorem
\ref{theo:affinemax} in Section \ref{section:technical}).
Particularly, the uniqueness statement fails in the case $n=p=r$.

In the next section, we recall the classification of optimal affine subspaces of nonsingular matrices, which will allow us to obtain an even more
constructive description of the affine subspaces of $\Mat_{n,p}(\F)$ with constant rank $r$ and dimension $d^{=}_r(n,p)$.

\subsection{A review of optimal spaces of nonsingular matrices}\label{section:optimal}

Here, we recall the main results of \cite{dSPgivenrank,dSPlargeaffinenonsingular}
on affine spaces of nonsingular matrices. The first result gives the greatest possible dimension for such an affine subspace.

\begin{theo}[See \cite{Quinlan,dSPgivenrank}]\label{theo:nonsingulardim}
Let $n>0$. Let $\calM$ be an affine subspace of $\Mat_n(\F)$ that is included in $\GL_n(\F)$.
Then
$$\dim \calM \leq \dbinom{n}{2}.$$
\end{theo}

As recalled earlier, the optimality of this result is exemplified by the space of all upper-triangular matrices with all diagonal entries equal to $1$.
In \cite{dSPlargeaffinenonsingular}, the structure of the optimal spaces was elucidated.
The description requires the notation for the \emph{joint} of subsets $\calX_1 \subseteq \Mat_{n_1}(\F),\dots,\calX_d \subseteq \Mat_{n_d}(\F)$,
denoted by $\calX_1 \vee \calX_2 \vee \cdots \vee \calX_d$ and
defined as the set of all $N \times N$ square matrices, with $N=\underset{k=1}{\overset{d}{\sum}} n_k$, of the form
$$\begin{bmatrix}
X_1 & [?] & & [?] \\
[0] & X_2 & \ddots &  \\
 & \ddots & \ddots & [?] \\
[0] & & [0] & X_d
\end{bmatrix}$$
where $X_1 \in \calX_1,\dots,X_d \in \calX_d$, and the blocks represented by question marks are arbitrary.

Finally, a matrix $P \in \GL_n(\F)$ is called \textbf{nonisotropic} whenever $\forall X \in \F^n \setminus \{0\}, \; X^TPX \neq 0$.
Denote by $\Mata_n(\F)$ the space of all alternating matrices of $\Mat_n(\F)$ (i.e.\ the skew-symmetric matrices with all diagonal entries zero).
In case $P \in \GL_n(\F)$ is nonisotropic, it is easily seen that $P+\Mata_n(\F)$ is an optimal subspace of $\Mat_n(\F)$
(note that $\forall X \in \F^n \setminus \{0\}, \; \forall A \in \Mata_n(\F), \; X^T (P+A)X=X^T PX \neq 0$). The following theorem is a form of
converse statement:

\begin{theo}[See \cite{dSPlargeaffinenonsingular}]
Let $n>0$. Assume that $|\F|>2$. Let $\calM$ be an optimal affine subspace of $\Mat_n(\F)$.
Then there exists a partition $n=n_1+\cdots+n_d$ into positive integers, together with nonisotropic matrices $P_1 \in \GL_{n_1}(\F),\dots,P_d \in \GL_{n_d}(\F)$
such that
$$\calM \sim (P_1+\Mata_{n_1}(\F)) \vee \cdots \vee (P_d+\Mata_{n_d}(\F)).$$
Moreover, the partition $(n_1,\dots,n_d)$ is uniquely determined by $\calM$, and for each $i \in \lcro 1,d\rcro$
the quadratic form $q_i : X \in \F^{n_i} \mapsto X^TP_iX$ is uniquely determined by $\mathcal{M}$ up to congruence
(two quadratic forms are congruent whenever they are equivalent up to multiplication with a nonzero scalar).
\end{theo}

In other words, every equivalence class of optimal affine subspaces of $\Mat_n(\F)$
is determined by a list of congruence classes of (nonzero) nonisotropic quadratic forms over $\F$.
In the special case $|\F|=2$, there is no known general form for the optimal affine subspaces of $\Mat_n(\F)$
(see \cite{dSPprimitiveF2} for examples).

From there, point (b) of Theorem \ref{theo:maxdim} can be reinterpreted as saying that, under the provision $|\F|>r+1$,
the equivalence classes of constant rank $r$ affine subspaces of $\Mat_{n}(\F)$ with dimension $\dbinom{r}{2}+r(n-r)$
are classified by lists $(q_1,\dots,q_m)$ of (nonzero) nonisotropic quadratic forms over $\F$, with respective ranks
$n_1,\dots,n_m$, such that $\underset{k=1}{\overset{m}{\sum}} n_k=r$,
and an integer $u \in \lcro 0,m\rcro$ (the corresponding partition $(s,t)$ of $r$ is defined by $t:=\underset{k=1}{\overset{u}{\sum}} n_k$).
With such data, we consider associated nonisotropic matrices $P_1,\dots,P_m$ and we construct the affine space
$$\bigl((P_1+\Mata_{n_1}(\F)) \vee \cdots \vee (P_u+\Mata_{n_u}(\F))\bigr) \wedge_{n,n} \bigl((P_{u+1}+\Mata_{n_{u+1}}(\F)) \vee \cdots \vee (P_m+\Mata_{n_m}(\F))\bigr).$$
Then, the equivalence class of the result determines the pair $(s,t)$ and the congruence classes of the nonisotropic quadratic forms
$q_1,\dots,q_m$.

In the remainder of the article, we will not need such a precise understanding of the optimal spaces, but we will need a basic result on them
(Lemma \ref{lemma:transitivity}). However, we will use theorem 3 from \cite{dSPlargeaffinerankbelow}, whose only known proof uses such a precise understanding.

Before we proceed, it is useful, in the prospect of the proof of Lemma \ref{lemma:transitivity}, to recall how affine subspaces of nonsingular matrices
are connected to a variation of spaces of nilpotent matrices. Let $V$ be a finite-dimensional vector space.
A \emph{linear} subspace $S$ of $\End(V)$ is said to have \textbf{trivial spectrum} whenever no element of $S$ has a non-zero eigenvalue in $\F$.
We adopt a similar definition for square matrices. Let $\calS \subseteq \Mat_n(\F)$ be an affine subspace, with translation vector space denoted by $\overrightarrow{S}$.
Let $A \in \calS \cap \GL_n(\F)$. Then for all $N \in \overrightarrow{S}$, the invertibility
of $A+tN$ for all $t \in \F$ is equivalent to the fact that $A^{-1}N$ has no nonzero eigenvalue in $\F$.
Therefore, $\calS \subseteq \GL_n(\F)$ if and only if $A^{-1}\overrightarrow{S}$ is a trivial spectrum subspace of $\Mat_n(\F)$.

\subsection{Strategy, structure of the article, and open problems}

Our basic proof strategy is similar to the one of Rubei in \cite{Rubei}. We take an element of an affine subspace $\calS$
with constant rank $r$, we put it in normalized form, and then we exploit the assumption that every element of $\calS$ has rank at most $r$ to obtain polynomial identities on the elements of the translation vector space of $\calS$, in a way that is similar to Flanders's seminal work on spaces of bounded rank matrices \cite{Flanders}.
The result of Theorem \ref{theo:dim} is then obtained by combining these identities with the main theorem of \cite{Quinlan,dSPgivenrank}
(recalled here as Theorem \ref{theo:nonsingulardim}), which deals with affine spaces of nonsingular square matrices.
The analysis of the spaces with maximal dimension is then obtained by refining the analysis and by using a key result
on optimal affine subspaces of square matrices, the Transitivity Lemma (see Lemma \ref{lemma:transitivity});
A powerful result from the study of affine spaces with rank bounded below \cite{dSPlargeaffinerankbelow} is also
used to wrap the proof up. The uniqueness statements in Theorem \ref{theo:maxdim} are obtained by interpreting the results in terms of sets
of bilinear forms and orthogonal subspaces, and then deferring to similar uniqueness statements for spaces of matrices with rank bounded below.

All in all, the tools and strategy are reminiscent of the ones of the recent \cite{dSPaffaltconstantrank}, with different details and additional complexity.

The article is laid out as follows: in Section \ref{section:technical}, we recall the main tools, developed by Flanders, Atkinson and Lloyd,
to study spaces of matrices with rank bounded above; then we recall a key result of \cite{dSPlargeaffinerankbelow} on large spaces of matrices with rank bounded below;
and we finish with the Transitivity Lemma for optimal affine spaces of nonsingular matrices. Section \ref{section:proofs} is devoted to the proof of
Theorem \ref{theo:dim} and of the existence statements in Theorem \ref{theo:maxdim}. The uniqueness statements in Theorem \ref{theo:maxdim}
are proved in Section \ref{section:uniqueness}, which can be read independently of the other sections.

Let us finish with an open problem. In both Theorems \ref{theo:dim} and \ref{theo:maxdim}, we require that $\F$ has more than $r+1$
elements. This provision is unavoidable in our proof, which relies on polynomial identities that are derived from upper bounds on the ranks of matrices,
and those identities do not hold anymore if $|\F|<r+2$. Yet, in Problems (2) and (3) cited in Section \ref{section:intro},
solutions have been found with little or no restriction on the underlying field (the fields with $2$ elements are generally a cause of concern, but a limited one as far as Problem (3)
is concerned). Hence, we can speculate that the result of Theorem \ref{theo:dim} remains true even for fields with small cardinality. But to prove this would require
a complete revolution in the methods, and at the present time we have no idea for an alternate way to tackle this problem.

\section{Technical preliminaries}\label{section:technical}

Our proof techniques essentially rely on block-matrix results from the theory of vector spaces of bounded rank matrices.
Primarily, we will use the following result, which we call the affine version of the Flanders-Atkinson lemma.
We point to \cite{AtkinsonPrim,dSPLLD1} for various proofs of it, and to \cite{dSPaffaltconstantrank} for a note on the affine version:

\begin{lemma}[Affine version of the Flanders-Atkinson lemma]
Let $n,p,r$ be integers with $0<r \leq \min(n,p)$. Assume that $|\F|>r+1$.
Let $J_r:=\begin{bmatrix}
I_r & 0 \\
0 & 0
\end{bmatrix}$ and $M=\begin{bmatrix}
A & C \\
B & D
\end{bmatrix}$ belong to $\Mat_{n,p}(\F)$, with $A \in \Mat_r(\F)$ and so on.
If $\rk(J_r+t M) \leq r$ for all $t \in \F$, then $D=0$ and $B A^kC=0$ for every integer $k \geq 0$.
\end{lemma}

\begin{theo}\label{theo:affinemax}
Let $n$ and $r$ be positive integers with $1 \leq r \leq n$. Assume that $|\F|>2$.
Let $\calT$ be an affine subspace of $\Mat_{n,r}(\F)$ in which every matrix has rank $r$, and
assume that $\codim_{\Mat_{n,r}(\F)} \calT \leq \frac{r(r+1)}{2}\cdot$
Then there exists an optimal affine subspace $\calM$ of $\Mat_r(\F)$ such that
$$\calT \sim \widetilde{\calM}^{(n)}:=\left\{\begin{bmatrix}
B \\
C
\end{bmatrix} \mid (B,C) \in \calM \times \Mat_{n-r,r}(\F)\right\}.$$
Moreover, the equivalence class of $\calM$ is uniquely determined by $\calT$.
\end{theo}

We finish with a result on optimal spaces, which will be used repeatedly in our proof of Theorem \ref{theo:maxdim}.

\begin{lemma}[Transitivity lemma for optimal spaces]\label{lemma:transitivity}
Let $\calS$ be an optimal affine subspace of $\Mat_n(\F)$ for some $n>0$.
Then there exists a linear hyperplane $H$ of $\F^n$, called an \textbf{$\calS$-transitivity exclusion}, such that
$$\forall X \in \F^n \setminus H, \quad \Vect_\F(\calS) X=\F^n.$$
\end{lemma}

There are two strategies for the proof of Lemma \ref{lemma:transitivity}. The first one is to rely on the
explicit description of the optimal affine spaces given in Section \ref{section:optimal}.
There is a simpler proof that relies upon a basic result on trivial spectrum subspaces, which we recall now:

\begin{Def}
Let $\calS$ be a linear subspace of $\End(V)$ for some vector space $V$.
A vector $x \in V \setminus \{0\}$ is called \textbf{$\calS$-adapted} whenever $\calS$ contains no nonzero operator with range $\F x$.
\end{Def}

The following result was proved in \cite{dSPgivenrank} thanks to a combinatorial method:

\begin{prop}[See proposition 10 in \cite{dSPgivenrank}]\label{prop:matrixadapted}
Let $\calM$ be a trivial spectrum linear subspace of $\Mat_n(\F)$.
Then at least one vector of the standard basis of $\F^n$ is $\calM$-adapted.
\end{prop}

\begin{cor}\label{cor:adapted}
Let $\calS$ be a trivial spectrum linear subspace of $\End(V)$ for some finite-dimensional vector space $V$.
Then there exists a linear hyperplane $H$ of $V$ which contains all the vectors $x \in V \setminus \{0\}$
that are not $\calS$-adapted.
\end{cor}

\begin{proof}[Proof of Corollary \ref{cor:adapted}]
Assuming otherwise, there would be a basis $(e_1,\dots,e_n)$ of $V$ none of whose terms is $\calS$-adapted.
Then, representing $\calS$ by a matrix space $\calM$ in that basis, we would obtain a contradiction with Proposition \ref{prop:matrixadapted}.
\end{proof}

We are now ready to prove Lemma \ref{lemma:transitivity}:

\begin{proof}[Proof of Lemma \ref{lemma:transitivity}]
Choose $A \in \calS$, and denote by $\overrightarrow{S}$ the translation vector space of $\calS$. Then the linear subspace $A^{-1} \overrightarrow{S}$ has trivial spectrum.
Let $X \in \F^n \setminus \{0\}$ be $A^{-1}\overrightarrow{S}$-adapted. We shall prove that $\F AX+\overrightarrow{S}X=\F^n$.
Without loss of generality, we can assume that $X$ is the last vector of the standard basis of $\F^n$.

Consider the subspace $W$ of $A^{-1}\overrightarrow{S}$ consisting of all its matrices with last column zero.
Then by the rank theorem $\dim (A^{-1}\overrightarrow{S} X)+\dim W=\dim A^{-1}\overrightarrow{S}$. Next, write every $M \in W$
as
$$M=\begin{bmatrix}
K(M) & [0]_{(n-1) \times 1} \\
[?]_{1 \times (n-1)} & 0
\end{bmatrix} \quad \text{with $K(M) \in \Mat_{n-1}(\F)$,}$$
and note that $K(M)$ has trivial spectrum.
Hence $K(W)$ is a trivial spectrum subspace of $\Mat_{n-1}(\F)$, leading to $\dim K(W) \leq \dbinom{n-1}{2}$ by Theorem \ref{theo:nonsingulardim}.
Finally, note that $M \in W \mapsto K(M)$ is injective because $X$ is $A^{-1}\overrightarrow{S}$-adapted.
Hence
$$\dim (A^{-1}\overrightarrow{S} X) =\dim \overrightarrow{S}-\dim W \geq \dbinom{n}{2}-\dbinom{n-1}{2}=n-1.$$
Finally $\F X \cap A^{-1}\overrightarrow{S}X=\{0\}$ because $A^{-1}\overrightarrow{S}$ has trivial spectrum. This yields
$\F A X+\overrightarrow{S} X=\F^n$, whence $\Vect_\F(\calS) X=\F^n$.

Hence $\Vect_\F(\calS) X=\F^n$ for every $A^{-1}\overrightarrow{S}$-adapted vector $X$. By combining this with Corollary \ref{cor:adapted}, we obtain the claimed result.
\end{proof}

\section{The greatest dimension for a constant rank affine subspace, and the structure of spaces with maximal dimension}\label{section:proofs}

\subsection{The inequality statement}

Here, we prove Theorem \ref{theo:dim}.

Let $\calS$ be an affine subspace of $\Mat_{n,p}(\F)$ with constant rank $r$.
Replacing $\calS$ with an equivalent subspace, we lose no generality in assuming that $\calS$
contains the matrix
$$J_r:=\begin{bmatrix}
I_r & [0]_{r \times (p-r)} \\
[0]_{(n-r) \times r} & [0]_{(n-r) \times (p-r)}
\end{bmatrix}.$$
We will write every matrix $M \in \Mat_{n,p}(\F)$ blockwise along the same format as $J_r$:
$$M=\begin{bmatrix}
A(M) & C(M) \\
B(M) & D(M)
\end{bmatrix}
\quad \text{with $A(M) \in \Mat_r(\F)$ and so on.}$$
Now, we consider the affine subspace $\calT$ of $\calS$ consisting of its matrices of the form
$$M=\begin{bmatrix}
A(M) & [0]_{r \times (p-r)} \\
[0]_{(n-r) \times r} & [0]_{(n-r) \times (p-r)}
\end{bmatrix}.$$
Obviously $A(\calT)$ is an affine subspace of $\Mat_r(\F)$ of constant rank $r$, and $\dim \calT=\dim A(\calT)$. Hence, by Theorem \ref{theo:nonsingulardim},
$$\dim \calT \leq \dbinom{r}{2}.$$
Next, we denote by $S$ the translation vector space of $\calS$.
Let $M \in S$. For all $\alpha \in \F$, the matrix $J_r+\alpha M$ has rank $r$ because it belongs to $\calS$. Since $|\F|>r+1$, we deduce from
the Flanders-Atkinson lemma that
$$B(M)\, C(M)=0 \quad \text{and} \quad D(M)=0.$$
Next, by the rank theorem applied to $M \in S \mapsto (B(M),C(M))$, we find
$$\dim \calS=\dim \calT+\dim W \quad \text{for} \quad W:=\bigl\{(B(M),C(M)) \mid M \in S\bigr\}.$$
To complete the proof of Theorem \ref{theo:dim}, it will suffice to prove that $\dim W \leq r(n-r)$.
Set
$$S':=\{M \in S : C(M)=0\}.$$
By the rank theorem we have
$$\dim W=\dim B(S')+\dim C(S).$$
Now, let $M \in S$ and $N \in S'$. Then $B(M+N)\,C(M+N)=0$ and $B(M)\,C(M)=0$. Since $C(N)=0$, this yields
$$B(N)\,C(M)=0.$$
Denote finally by
$$U:=\sum_{M \in S} \im C(M) \subseteq \F^r$$
the sum of the column spaces of the elements of $C(S)$,
and set
$$s:=\dim U \quad \text{and} \quad t:=r-s,$$
so that
$$\dim C(S) \leq s\,(p-r).$$
As all the elements of $B(S')$ vanish everywhere on $U$, we find
$$\dim B(S') \leq t\,(n-r).$$
Hence, since $n \geq p$,
$$\dim W \leq t\,(n-r)+s\,(p-r) \leq r\,(n-r).$$
This yields
$$\dim \calS \leq \dbinom{r}{2}+r\,(n-r),$$
thereby completing the proof of Theorem \ref{theo:dim}.

\subsection{The case of equality (I)}\label{section:equality1}

\label{page:startequality}

Now, we assume that $\dim \calS=\dbinom{r}{2}+r(n-r)$. By analyzing the previous proof, it follows that:
\begin{enumerate}[(i)]
\item $\dim \calT=\dbinom{r}{2}$;
\item $C(S)$ is the set of all matrices of $\Mat_{r,p-r}(\F)$ with column space included in $U$;
\item $B(S')$ is the set of all matrices of $\Mat_{n-r,r}(\F)$ whose nullspace includes $U$.
\end{enumerate}
Moreover, if $n>p$ then $s>0$ would lead to
$$t\,(n-r)+s\,(p-r)<t\,(n-r)+s\,(n-r)=r\,(n-r),$$
contradicting the assumption that $\dim \calS=\dbinom{r}{2}+r(n-r)$.
Hence $s=0$ if $n>p$.

Now, as a consequence of point (i) above, $A(\calT)$ is an optimal affine subspace of $\Mat_r(\F)$.

We proceed and apply the Flanders-Atkinson lemma once more, and this time around we start to apply the second set of identities with $k=1$.

\begin{claim}\label{claim:k=1Atk}
Let $M \in S$. Then $B(M)\,A'\,C(M)=0$ for all $A' \in A(\calT)$.
\end{claim}

\begin{proof}
The case $A'=I_r$ is already known. Denote by $T$ the translation vector space of $\calT$, and let $N \in T$.
Applying the Flanders-Atkinson lemma to $M$ and $M+N$, we find
$$B(M)\,A(M+N)\,C(M)=0 \quad \text{and} \quad B(M)\,A(M)\,C(M)=0.$$
By subtracting we deduce that $B(M)\,A(N)\,C(M)=0$. Hence $B(M)\,A(J_r+N)\,B(M)=B(M)\,C(M)+B(M)\,A(N)\,C(M)=0$.
This holds for all $N \in T$, and hence $B(M)\,A'\,C(M)=0$ for all $A' \in \calT$.
\end{proof}

\begin{claim}\label{claim:Uinvariant}
The subspace $U$ is invariant under all the elements of $A(\calT)$.
\end{claim}

\begin{proof}
Fix $A' \in A(\calT)$. Polarizing $\forall M \in S, \; B(M) A' C(M)=0$,
we obtain $B(M) A' C(N)+B(N) A' C(M)=0$ for all $(M,N)\in S^2$.
In particular
\begin{equation}\label{eq:superortho}
\forall M \in S', \; \forall N \in S, \; B(M) A' C(N)=0.
\end{equation}
Assume now that $U$ is not invariant under $A'$. Then, choose $X \in U \setminus \{0\}$ such that $A'X \not\in U$.
We can choose $N \in S$ with all columns of $C(N)$ equal to $X$ (by point (ii) above), and we can choose
$M \in S'$ such that $B(M)A'X \neq 0$ (by point (iii) above). This would lead to $B(M)A'C(N) \neq 0$, contradicting
identity \eqref{eq:superortho}. Hence $U$ is invariant under all the elements of $A(\calT)$.
\end{proof}

Next, replacing $\calS$ with an equivalent subspace, we can refine the situation to the point where $U=\F^s \times \{0_{r-s}\}$
(still assuming that $\calS$ contains $J_r$). In that situation, Claim \ref{claim:Uinvariant} shows that the elements of $A(\calT)$ have the form
$$\begin{bmatrix}
[?]_{s \times s} & [?]_{s \times t} \\
[0]_{t \times s} & [?]_{t \times t}
\end{bmatrix}.$$
We write such matrices as
$$A(M)=\begin{bmatrix}
A_{1,1}(M) & A_{1,2}(M) \\
[0]_{t \times s} & A_{2,2}(M)
\end{bmatrix}$$
and we note that $A_{1,1}(\calT)$ and $A_{2,2}(\calT)$ are affine subspaces of nonsingular matrices of $\Mat_s(\F)$ and $\Mat_t(\F)$, respectively.
Hence
$$\dim A(\calT) \leq \dim A_{1,1}(\calT)+\dim A_{2,2}(\calT)+st \leq \dbinom{s}{2}+\dbinom{t}{2}+st=\dbinom{r}{2}=\dim A(\calT),$$
and we deduce that $A_{1,1}(\calT)$ and $A_{2,2}(\calT)$ are optimal and that
$$A(\calT)=A_{1,1}(\calT) \vee \calA_{2,2}(\calT).$$

Next, for all $M \in S$ we write
$$B(M)=\begin{bmatrix}
B_1(M) & B_2(M)
\end{bmatrix} \quad \text{with $B_1(M) \in \Mat_{n-r,s}(\F)$ and $B_2(M) \in \Mat_{n-r,t}(\F)$,}$$
and we note that the mapping
$B_2 : S' \rightarrow \Mat_{n-r,t}(\F)$ is surjective and the restriction of $B_1$ to $S'$ is zero.
Finally, for all $M \in S$ we write
$$C(M)=\begin{bmatrix}
C_1(M) \\
[0]_{t \times (p-r)}
\end{bmatrix} \quad \text{with $C_1(M)\in \Mat_{s,p-r}(\F)$,} $$
and we note that the mapping $C_1 : S \rightarrow \Mat_{s,p-r}(\F)$ is surjective.

We arrive now at a key step, which uses the Transitivity Lemma from Section~\ref{section:technical}.

\begin{claim}\label{claim:B1nul}
One has $B_1(M)=0$ for all $M \in S$.
\end{claim}

\begin{proof}
The result is trivial if $s=0$. Now, we assume that $s>0$.

By applying the result of Claim \ref{claim:k=1Atk}, we now find
$$\forall M \in S, \; \forall K \in A_{1,1}(\calT), \; B_1(M)\, K\, C_1(M)=0.$$
By linearity with respect to $K$, it follows that
\begin{equation}\label{eq:orthoB1C1}
\forall M \in S, \; \forall K \in \Vect_\F(A_{1,1}(\calT)), \; B_1(M)\, K\, C_1(M)=0.
\end{equation}
Remember that $A_{1,1}(\calT)$ is an optimal affine subspace of $\Mat_s(\F)$ (with $s>0$),
and then apply Lemma \ref{lemma:transitivity} to find a linear hyperplane $H$ of $\F^s$
such that $\{K X \mid K \in \Vect_\F(A_{1,1}(\calT))\}=\F^s$ for all $X \in \F^s \setminus H$.

Now, let $M \in S$ and assume that the first column $X$ of $C_1(M)$ is outside of $H$.
Then identity \eqref{eq:orthoB1C1} yields that all the rows of $B_1(M)$
are orthogonal to all the vectors of $\{K X \mid K \in \Vect_\F(A_{1,1}(\calT))\}$, and hence $B_1(M)=0$.
Because of point (ii) on page \pageref{page:startequality}, the set $S_1$ of all $M \in S$ for which the first column of $C_1(M)$ is in $H$ is a proper linear subspace
of $S$, and hence $S \setminus S_1$ is a spanning subset of $S$. Since the linear mapping $B_1$ vanishes everywhere on this subset,
we conclude that $B_1=0$.
\end{proof}

Now, we have linear mappings
$$F : \Mat_{n-r,t}(\F) \longrightarrow \Mat_{t,s}(\F), \quad
G : \Mat_{s,p-r}(\F) \longrightarrow \Mat_{t,s}(\F)$$
and
$$(Y,Z) \in \Mat_{n-r,t}(\F) \times \Mat_{s,p-r}(\F) \longmapsto M_{Y,Z} \in S$$
such that, for every $(Y,Z) \in \Mat_{n-r,t}(\F) \times \Mat_{s,p-r}(\F)$,
$$M_{Y,Z}=\begin{bmatrix}
[?]_{s \times s} & [?]_{s \times t} & Z \\
F(Y)+G(Z) & [?]_{t \times t} & [0]_{t \times (p-r)} \\
[0]_{(n-r) \times s} & Y & [0]_{(n-r)\times (p-r)}
\end{bmatrix}.$$
Moreover, every matrix of $S$ is the sum of a matrix of type $M_{Y,Z}$ with a matrix of $T$, the translation vector space of $\calT$.

\begin{claim}
The mappings $F$ and $G$ are zero.
\end{claim}

\begin{proof}
The result is trivial if $s=0$ or $t=0$, so we assume that $s>0$ and $t>0$.

Applying the Flanders-Atkinson lemma once more (with $k=1$) yields
$$\forall (Y,Z) \in \Mat_{n-r,t}(\F) \times \Mat_{s,p-r}(\F), \; Y (F(Y)+G(Z))\,Z=0.$$
Because $|\F|>2$, we can split this identity (after applying it to $(\alpha Y,Z)$ with an arbitrary $\alpha \in \F$) into :
$$\forall (Y,Z) \in \Mat_{n-r,t}(\F) \times \Mat_{s,p-r}(\F), \; YF(Y)Z=0 \quad \text{and} \quad YG(Z)Z=0,$$
and hence
\begin{equation}\label{eq:Atk1}
\forall (Y,Z) \in \Mat_{n-r,t}(\F) \times \Mat_{s,p-r}(\F), \; YF(Y)=0 \quad \text{and} \quad G(Z)Z=0.
\end{equation}
Unfortunately, this is still insufficient if $p=r+1$!
So, we go further in the Flanders-Atkinson lemma and take $k=2$ in its second set of identities.
Let $(Y,Z) \in \Mat_{n-r,t}(\F) \times \Mat_{s,p-r}(\F)$ and $N \in T$. Applying the Flanders-Atkinson lemma
to $M_{Y,Z}+\alpha N$ for an arbitrary $\alpha \in \F$,
we find
$$\forall \alpha \in \F, \; B(M_{Y,Z})\, (A(M_{Y,Z})+\alpha\,A(N))^2\, C(M_{Y,Z})=0.$$
Since $|\F|>2$ and the left-hand side is formally a polynomial of degree at most $2$ in $\alpha$, we extract the coefficient on $\alpha$ to obtain
$$B(M_{Y,Z})\, \bigl(A(M_{Y,Z})A(N)+A(N)A(M_{Y,Z})\bigr)\, C(M_{Y,Z})=0,$$
which can be rewritten as
$$Y (F(Y)+G(Z))A_{1,1}(N) Z+Y A_{2,2}(N) (F(Y)+G(Z))\,Z=0.$$
Thanks to \eqref{eq:Atk1}, we simplify this as
$$Y G(Z) A_{1,1}(N)Z+Y A_{2,2}(N) F(Y)Z=0.$$
Now, remembering that $A(\calT)=A_{1,1}(\calT) \vee  A_{2,2}(\calT)$, we derive by varying $N$ that
$$\forall K \in A_{1,1}(T), \; Y G(Z) K Z=0 \quad \text{and} \quad
\forall K' \in A_{2,2}(T), \; Y K' F(Y) Z=0.$$
Again, by varying $Y$ in the first identity and $Z$ in the second one, we end up with the simplified identities
$$\forall K \in A_{1,1}(T),\; \forall Z \in \Mat_{s,p-r}(\F), \; G(Z)KZ=0.$$
and
$$\forall K' \in A_{2,2}(T), \; \forall Y \in \Mat_{n-r,t}(\F), \; Y K' F(Y)=0.$$
Note that $J_r \in \calT$ leads to $I_s \in A_{1,1}(\calT)$ and $I_t \in A_{2,2}(\calT)$.
Hence, by combining the previous two identities with \eqref{eq:Atk1} we deduce that
$$\forall K \in \Vect_\F(A_{1,1}(\calT)),\; \forall Z \in \Mat_{s,p-r}(\F), \; G(Z)KZ=0$$
and
$$\forall K' \in \Vect_\F(A_{2,2}(\calT)), \; \forall Y \in \Mat_{n-r,t}(\F), \; Y K' F(Y)=0.$$
By using the Transitivity Lemma just like in the proof of Claim \ref{claim:B1nul}, we obtain that $G(Z)=0$ for all $Z \in \Mat_{s,p-r}(\F)$.
Finally, note by transposing that
$$\forall K' \in \Vect_\F(A_{2,2}(\calT)^T), \; \forall Y \in \Mat_{t,n-r}(\F), \; F(Y^T)^T K' Y^T=0$$
and that $A_{2,2}(\calT)^T$ is an optimal affine subspace of $\Mat_t(\F)$. Hence, by applying the Transitivity Lemma once more, we obtain
$F(Y^T)^T=0$ for all $Y \in \Mat_{t,n-r}(\F)$, and finally $F=0$.
\end{proof}

\subsection{The case of equality (II)}

As every matrix of $S$ is the sum of a matrix of type $M_{Y,Z}$ and of a matrix of $T$, we obtain that every $M \in S$ has the form
$$\begin{bmatrix}
[?]_{s \times s} & [?]_{s \times t} & [?]_{s \times (p-r)} \\
[0]_{t \times s} & [?]_{t \times t} & [0]_{s \times (p-r)} \\
[0]_{(n-r) \times s} & [?]_{(n-r) \times t} & [0]_{(n-r)\times (p-r)}
\end{bmatrix}.$$

At this point, we completely change the space: by permuting columns, we find that
$\calS$ is equivalent to a space $\calU$ of matrices in which every matrix has the form
$$M=\begin{bmatrix}
[?]_{s \times t} & L(M) \\
K(M) & [0]_{(n-s) \times (p-t)}
\end{bmatrix} \quad \text{with $K(M) \in \Mat_{n-s,t}(\F)$ and $L(M) \in \Mat_{s,p-t}(\F)$.}$$
And now we will conclude thanks to Theorem \ref{theo:affinemax}.
First of all, for all $M \in \calU$ we have
$$\rk M \leq t+\rk L(M) \leq t+s \quad \text{and} \quad \rk M \leq s+\rk K(M)\leq s+t,$$
which yields $\rk L(M)=s$ and $\rk K(M)=t$.
By Theorem \ref{theo:dim}, this yields $\dim L(\calU) \leq \dbinom{s}{2}+s(p-r)$ and $\dim K(\calU) \leq \dbinom{t}{2}+t(n-r)$.
Moreover
$$\dim \calU \leq st+\dim K(\calU)+\dim L(\calU) \leq \dbinom{r}{2}+r(n-r),$$
and it follows:
\begin{enumerate}[(i)]
\item That $\dim L(\calU)=\dbinom{s}{2}+s(p-r)$;
\item That $\dim K(\calU)=\dbinom{t}{2}+t(n-r)$;
\item And that $\calU$ is the set of all matrices of the form
$$\begin{bmatrix}
[?]_{s \times t} & L' \\
K' & [0]_{(n-s) \times (p-t)}
\end{bmatrix} \quad \text{with $K' \in K(\calU)$ and $L' \in L(\calU)$.}$$
\end{enumerate}

Now, we can finally apply Theorem \ref{theo:affinemax} to $K(\calU)$ and $L(\calU)^T$. This yields
respective optimal affine subspaces $\calM$ and $\calN$ of $\Mat_t(\F)$ and $\Mat_s(\F)$ and invertible matrices
$P_1 \in \GL_{n-s}(\F)$ and $Q_2 \in \GL_{p-t}(\F)$ such that
$K(\calU)=P_1 \widetilde{\calM}^{(n-s)}$ and $L(\calU)=(\widetilde{\calN}^{(p-t)})^T Q_2$.
Replacing $\calU$ with $(I_s \oplus P_1)^{-1} \calU (I_t \oplus Q_2)^{-1}$, we can further reduce the situation to the one where
$K(\calU)=\widetilde{\calM}^{(n-s)}$ and $L(\calU)=(\widetilde{\calN}^{(p-t)})^T$.
And from there we have the equality
$$\calS \sim \calU=\calM \wedge_{n,p} \calN^T.$$
If $n>p$ then $s=0$ (as seen at the start of Section \ref{section:equality1}) and hence
$$\calS \sim \calU=\widetilde{\calM}^{(n,p)}.$$

This completes the proof of the existence statements in Theorem \ref{theo:maxdim}.

\section{Classification of spaces with maximal dimension}\label{section:uniqueness}

In this final part, we prove the uniqueness statements of Theorem \ref{theo:maxdim}.
We will discuss the situations $n=p$ and $n>p$ separately, beginning with the latter.
In both cases, we will use geometric considerations, by seeing $\calS$
as representing a set of bilinear forms on $\F^n \times \F^p$.

\subsection{The case $n>p$}

Let $\calM_1$ and $\calM_2$ be optimal affine subspaces of $\Mat_r(\F)$, and assume that
$\widetilde{\calM_1}^{(n,p)} \sim \widetilde{\calM_2}^{(n,p)}$.
We will prove that $\widetilde{\calM_1}^{(n)} \sim \widetilde{\calM_2}^{(n)}$ (see the notation in Theorem \ref{theo:affinemax})
and then Theorem \ref{theo:affinemax} will yield that $\calM_1 \sim \calM_2$.

To prove the claimed result, we start by noting that all the elements of $\widetilde{\calM_1}^{(n,p)}$
have $\Vect(e_{r+1},\dots,e_p)$ as nullspace, where $(e_1,\dots,e_p)$ denotes the standard basis of $\F^p$.
Indeed, it is clear that all such matrices have $\Vect(e_{r+1},\dots,e_p)$ included in their nullspace, and the equality follows from the fact that every matrix in
$\widetilde{\calM_1}^{(n,p)}$ has rank $r$. Now, let $P\in \GL_n(\F)$ and $Q \in \GL_p(\F)$ be such that
$\widetilde{\calM_2}^{(n,p)}=P \widetilde{\calM_1}^{(n,p)} Q$.
Applying the previous remark, we find that $Q$ leaves $\Vect(e_{r+1},\dots,e_p)$ invariant. Hence
$$Q=\begin{bmatrix}
Q_1 & [0]_{r \times (p-r)} \\
[?]_{(p-r) \times r} & [?]_{(p-r) \times (p-r)}
\end{bmatrix}$$
for some $Q_1 \in \GL_r(\F)$. Then, by extracting the first $p$ columns we find that
$$\widetilde{\calM_2}^{(n)}=P \widetilde{\calM_1}^{(n)} Q_1,$$
which is the claimed equivalence. This completes the proof.

\subsection{The case $n=p$}

Now, we assume that $n=p$.
We start by considering a partition $r=s+t$, with $s \geq 0$ and $t \geq 0$,
together with optimal affine subspaces $\calM \subseteq \Mat_t(\F)$ and $\calN \subseteq \Mat_s(\F)$, and we consider the
space $\calS=\calM \wedge_{n,n} \calN$.
For subsets $\calX$ and $\calY$ of $\F^n$, we will write $\calX \underset{\calS}{\bot} \calY$ to mean that
$$\forall (X,Y) \in \calX \times \calY, \; \forall M \in \calS, \; X^T M Y=0,$$
i.e.\ $\calX$ is left-orthogonal to $\calY$ under all the elements of $\calS$, seen as bilinear forms on $\F^n$.

Denote by $(e_1,\dots,e_n)$ the standard basis of $\F^n$.
The key is to consider the spaces $F:=\Vect(e_{s+1},\dots,e_n)$ and $G:=\Vect(e_{t+1},\dots,e_n)$
and to note that $F \underset{\calS}{\bot} G$. We shall now see that $(F,G)$ is the only such pair for which $\dim F+\dim G=2n-r$.

\begin{claim}\label{claim:orthogonality}
Let $F',G'$ be linear subspaces of $\F^n$ such that $F' \underset{\calS}{\bot} G'$ and $\dim F'+\dim G'=2n-r$.
Then $F'=F$ and $G'=G$.
\end{claim}

\begin{proof}
The assumption on the dimension yields that $\dim F'=n-s'$ and $\dim G'=n-t'$ for some pair $(s',t')$ of non-negative integers such that $s'+t'=r$.
Next, note by linearity that $F' \underset{\Vect(\calS)}{\bot} G'$.
Finally, note by basic orthogonality theory that for every rank $r$ matrix $A$, if we have linear subspaces
$H$ and $H'$ of $\F^n$ such that $H \underset{\{A\}}{\bot} H'$, then $\dim H+\dim H'\leq 2n-r$, and if equality occurs then $H'$ is the right-$A$-orthogonal of $H$, that is
the set of all vectors $Y \in \F^n$ such that $\forall X \in H, \; X^T AY=0$.

We will start by proving that $F' \subseteq F$ or $G' \subseteq G$.
To do so, we take arbitrary vectors $X \in F'$ and $Y \in G'$, which we write
$X=\begin{bmatrix}
X_1 \\
[?]_{(n-s) \times 1}
\end{bmatrix}$ and
$Y=\begin{bmatrix}
Y_1 \\
[?]_{(n-t) \times 1}
\end{bmatrix}$
with $X_1 \in \F^s$ and $Y_1 \in \F^t$.
Note that, for all $A \in \Mat_{s,t}(\F)$, the matrix
$$\begin{bmatrix}
A & [0]_{s \times (n-t)} \\
[0]_{(n-s) \times t} &  [0]_{(n-s) \times (n-t)}
\end{bmatrix}$$
belongs to $\Vect_\F(\calS)$, which leads to $X_1^T A Y_1=0$.
Varying $A$ shows that $X_1=0$ or $Y_1=0$.
And this proves that $F' \subseteq F$ or $G' \subseteq G$.

Now, if $F'\not\subseteq F$ we have $G' \subseteq G$ and transposing $\calS$ takes us back to the first case.
So, in the remainder of the proof we will only consider the case where $F' \subseteq F$.

Next, we prove that $G \subseteq G'$.
This is an easy consequence of basic orthogonality theory. Indeed, in $\calS$ we can pick a rank $r$ matrix $M$.
By an early remark in this proof we find that $G$ is the right-$M$-orthogonal of $F$, and $G'$ is the right-$M$-orthogonal of $F'$.
Since $F' \subseteq F$ this yields $G \subseteq G'$.

Hence if $G'=G$ then $F=F'$ because $\dim G+\dim F=\dim G'+\dim F'$.
Now, we assume that $G \subsetneq G'$ in the remainder of the proof. Then we can split $G'=G''\oplus G$ for a nonzero linear subspace $G''$ of $\Vect(e_1,\dots,e_t)$.
We will use the fact that $F' \underset{\calS}{\bot} G''$. By identifying $F'$ and $G''$ with respective linear subspaces of $F'_1$ and $G''_1$ of $\F^{n-s}$ and $\F^t$,
we obtain that $F'_1$ is left-$\widetilde{\calM}^{(n-s)}$-orthogonal to $G''_1$, and $G''_1 \neq \{0\}$.
Next, we claim that $F'_1 \subseteq \F^t \times \{0_{n-r}\}$.

To see this, choose $Y \in G''_1 \setminus \{0\}$, let $B \in \Mat_{n-r,t}(\F)$ and consider the matrix $M=\begin{bmatrix}
[0]_{t \times t} \\
B
\end{bmatrix}$, which belongs to $\Vect_\F(\widetilde{\calM}^{(n-s)})$. Hence $\forall X \in F'_1, \; X^T MY=0$.
By varying $B$ we see $\forall X \in F'_1, \; \forall Y \in \{0_t\} \times \F^{n-r}, \; X^T Y=0$, which yields $F'_1 \subseteq \F^t \times \{0_{n-r}\}$.

Now, we can finally conclude: by identifying $F'_1$ with a subspace $F''_1$ of $\F^t$, we see that $F''_1 \underset{\calM}{\bot} G''_1$.
Since $\calM$ contains at least one rank $t$ matrix, this yields
$\dim F''_1+\dim G''_1 \leq t$, which reads $(n-s')+(n-t')-(n-t) \leq t$, that is $n-r \leq 0$, and this contradicts the assumption that $n>r$.
We conclude that $F=F'$ and $G=G'$, as claimed.
\end{proof}

Now, we can conclude. Let $r=s'+t'$ be a partition of $r$ into non-negative integers, let $\calM' \subseteq \Mat_{t'}(\F)$ and $\calN' \subseteq \Mat_{s'}(\F)$, and consider the space $\calS':=\calM' \wedge_{n,n} \calN'$. Set $F'':=\Vect(e_{s'+1},\dots,e_n)$ and $G'':=\Vect(e_{t'+1},\dots,e_n)$.
Assume that there exist $P,Q$ in $\GL_n(\F)$ such that $\calS'=P\,\calS\,Q$.
Then for $F':=P^T F''$ and $G':=Q G''$ we have $\dim F'+\dim G'=2n-r$ and $F' \underset{\calS}{\bot} G'$.
By Claim \ref{claim:orthogonality}, we deduce that $F'=F$ and $G'=G$, and further that $s=s'$ and $t=t'$ by comparing the dimensions.
Hence $F=P^TF$ and $G=QG$. It follows that
$$P=\begin{bmatrix}
P_1 & [?]_{s \times (n-s)} \\
[0]_{(n-s) \times s} & P_2
\end{bmatrix} \quad \text{and} \quad
Q=\begin{bmatrix}
Q_1 & [0]_{t \times (n-t)} \\
[?]_{(n-t) \times t} & Q_2
\end{bmatrix}$$
for some $P_1 \in \GL_s(\F)$, $Q_1 \in \GL_t(\F)$, $P_2 \in \GL_{n-s}(\F)$ and $Q_2 \in \GL_{n-t}(\F)$.
Then, by extracting the blocks from the identity $\calS'=P\,\calS\,Q$, we obtain
$$\widetilde{\calM'}^{(n)}=P_2\, \widetilde{\calM}^{(n)}\, Q_1 \quad \text{and} \quad
\Bigl(\widetilde{(\calN')^T}^{(n)}\Bigr)^T=P_1\, \Bigl(\widetilde{(\calN)^T}^{(n)}\Bigr)^T\, Q_2.$$
From there, we apply the uniqueness statement in Theorem \ref{theo:affinemax} to obtain that
$\calM' \sim \calM$ and $(\calN')^T \sim (\calN)^T$, and by transposing the latter we conclude that $\calN' \sim \calN$.
This completes the proof of Theorem \ref{theo:maxdim}.

\end{document}